\documentclass[a4paper,12pt]{amsart}
\usepackage{amssymb,amscd,amsmath}
\usepackage[dvips]{color}

\theoremstyle{plain}
\newtheorem{thm}{Theorem}[section]

\newtheorem{lem}[thm]{Lemma}

\theoremstyle{definition}

\theoremstyle{remark}
\newtheorem{rem}{Remark}[section]
\newtheorem{ex}[rem]{Example}

\newcommand{\forme}[1]{}
\begin{document}
%============= title ========================%%
\title{Characterization of finite colored spaces with certain conditions}

\author{Mitsugu Hirasaka}
\author{Masashi Shinohara}

\address{Department of Mathematics, College of Sciences, Pusan National University,
63 Beon-gil 2, Busandaehag-ro, Geumjung-gu, Busan 609-735, Korea.}

\address{Department of Education, Faculty of Education, Shiga University, 2-5-1 Hiratsu, Otsu,  SHIGA 520-0862, Japan}
\email{hirasaka@pusan.ac.kr}
\email{shino@edu.shiga-u.ac.jp}
%\email{}

\date{\today}

\thanks{}
\thanks{}
\thanks{}

\maketitle

%============================================%%

%%%%%%%%%%%%%%%%%%%%%%%%%%%%%%%%%%%%%%%%%%%%%%%%%%%
\begin{abstract}
A \textit{colored space} is the pair $(X,r)$
of a set $X$ and a function $r$ whose domain is $\binom{X}{2}$.
Let $(X,r)$ be a finite colored space and $Y,Z\subseteq X$.
We shall write $Y\simeq_r Z$ if there exists a bijection $f:Y\to Z$ such that
$r(U)=r(f(U))$ for each $U\in\binom{Y}{2}$.
We denote the numbers of equivalence classes with respect to $\simeq_r$
contained in $\binom{X}{2}$ and $\binom{X}{3}$ by $a_2(r)$ and $a_3(r)$, respectively.
In this paper we prove that $a_2(r)\leq a_3(r)$ when $5\leq |X|$, and show
what happens when the equality holds.
\end{abstract}
%%%%%%%%%%%%%%%%%%%%%%%%%%%%%%%%%%%%%%%%%%%%%%%%%%%

%%%%%%%%%%%%%%%%%%%%%%%%%%%%%%%%%%%%%%%%%%%%%%%%%%%
\section{Introduction}
A \textit{colored space} is the pair $(X,r)$ of a set $X$ and a function $r$
whose domain is $\binom{X}{2}$,
where we denote by $\binom{X}{k}$ the set of $k$-subsets of $X$ for a positive integer $k$.
Notice that $r$ can be identified with an edge-coloring of the complete graph of $X$, which
leads us to call the elements of $\mathsf{Im}(r)$ \textit{colors} where
$\mathsf{Im}(r)$ is the image of $r$.
From this point of view simple graphs induce colored spaces with at most two colors.
On the other hand, metric spaces also induce colored space whose colors are defined by
its metric function. Thus, we observe that colore spaces covers quite general objects
of combinatorics and geometry.

Let $(X,r)$ be a colored space.
Then each subset $Y$ of $X$ induces a colored space $(Y,r_Y)$,
called a \textit{subspace}, where $r_Y$ is the restriction
of $r$ to $\binom{Y}{2}$. For $Y,Z\subseteq X$ we say that
$Y$ is \textit{isometric} to $Z$ if there exists a bijection $f:Y\to Z$ with $r_Y=r_Z\circ f$,
and we shall write
\[Y\simeq_r Z\,\,\mbox{if $Y$ is isometric to $Y$}.\]
For a positive integer $k$ we denote the class of isometric subspaces of size $k$
by $A_k(r)$, i.e.,
\[A_k(r)=\left\{[Y]\mid Y\in \binom{X}{k}\right\}\,\, \mbox{where}\,\, [Y]=\{Z\mid Y\simeq_r Z\}.\]
We denote the cardinality of $A_k(r)$ by $a_k(r)$.
One may notice that $A_1(r)$ is a singleton, $A_2(r)$ can be identified with
$\mathsf{Im}(r)$,
and $A_k(r)$ is a finite set if $\mathsf{Im}(r)$  is a finite set.
Thus, the following sequence of positive integers are defined when $|X|=n$:
\[(a_1(r),a_2(r),\ldots, a_{n}(r)),\]
which is called
the \textit{isometric sequence} of $(X,r)$.
As mentioned before, $a_1(r)=1$, $a_2(r)=|\mathsf{Im}(r)|$, and it is easy to obtain
the following:
\begin{ex}\label{ex:2}
For each finite colored space $(X,r)$ with $n=|X|$ we have
\[1\leq a_2(r)\leq \binom{n}{2}.\]
One of the extremal cases has the isometric sequence $(1,1,\ldots, 1)$ and another
has the isometric sequence
\[\left(1,\binom{n}{2},\binom{n}{3},\ldots, \binom{n}{k},\ldots,\binom{n}{n-1},1\right).\]
\end{ex}
Notice that, for $Y,Z\in \binom{X}{3}$,
$Y\simeq_r Z$ if and only if $(r(U)\mid U\in \binom{Y}{2})$ coincides with
$(r(V)\mid V\in \binom{Z}{2})$ as multi-sets, so that we obtain from
the rule of combinations with repetitions that
\[1\leq a_3(r)\leq \binom{a_2(r)+2}{3}.\]
One of the optimal cases appears in four points in a Euclidean space like
the vertices of a rectangle, a tetrahedron with congruent faces or a regular simplex.
On the other hand, another can be attained when you have enough many points, and
one may notice that such colored spaces are conventional and not so curious.
In this paper we focus on how isometric sequences changes.
So far we have seen only isometric sequences changes unimodally, and
it is conjectured to be true for every isometric sequence.
The first step toward this conjecture must be to say what happens when
$a_2(r)>a_3(r)$, that is exactly what is stated in the following our main theorem:
\begin{thm}\label{thm:a3}
For each finite colored space $(X,r)$ with $5\leq |X|$ we have $a_2(r)\leq a_3(r)$.
\end{thm}
Namely, if $a_2(r)>a_3(r)$, then $|X|\leq 4$, and hence the isometric sequence is unimodal.
In Theorem~\ref{thm:a3} it is natural to ask when the equality holds.
For example,
the vertices of a regular octahedron and a regular hexagon in a Euclidean space induce the isometric sequences, respectively,
\[(1,2,2,2,1,1), (1,3,3,3,1,1).\]
In \cite{hs} the authors classify all colored spaces with $a_2(r)=a_3(r)\leq 3$ and $5\leq |X|$
up to isomorphism (see Section~2 for the definition)\footnote{In \cite{hs}
only finite metric spaces are discussed.
However, all results on the classifications of finite metric spaces
can be translated to those on colored spaces, since the classification is done up to isomorphisms.}
Cotinued to these works we solved this problem completely in the following:
\begin{thm}\label{thm:1}
Every finite colored space $(X,r)$ with $4\leq a_2(r)=a_3(r)$ and $9\leq |X|$ induces the partition of $\binom{X}{2}$ given in one of the following:
\begin{enumerate}
\item disjoint matchings and the remaining where disjoint matchings mean matchings whose union is also a matching;

\item a disjoint union of two cliques, disjoint matchings between the two cliques and the remaining;

\item a singleton of one edge $U$, the edges from one end of the edge to $X\setminus U$,
the edges from the other end to $X\setminus U$, and the remaining.
\end{enumerate}
\end{thm}
\begin{rem}
We have $a_2(r)=a_3(r)=4$ if (iii) happens in Theorem~\ref{thm:1}.
\end{rem}

In relation to this topic we refer some articles on distance sets (see \cite{BBS}, \cite{ES66}, \cite{LRS}
and \cite{Li}). For a positive integer $s$ we say that a finite subset $X$
in a Euclidean space $\mathbb{R}^N$
is an $s$-\textit{distance set} if $a_2(r)=2$ where $r:\binom{X}{2}\to \mathbb{R}_{\geq 0}$ is
is defined by $r(\{x,y\})=||x-y||$. In \cite{BBS}, an upper bound for $|X|$ is
given by a function on $s$ and $N$. In \cite{LRS} and \cite{Li} they show $2$-distance sets in $\mathbb{R}^N$ which attains the upper bound. In \cite{ES66} they show some criterion
to embed finite colored spaces $(X,r)$ with $a_2(r)=2$ into
a Euclidean space of dimension less than $|X|-1$.
In such a way studies on distance sets focus on finite subsets in a Euclidean spaces with
certain optimal conditions. From this point of view Theorem~\ref{thm:1} gives a characterization
of finite subsets $X$ of a Euclidean space whose distances corresponds to
the congruence classes of triangles derived from the elements in $\binom{X}{3}$.

In Section~2 we prepare some notation on colored spaces, and in Section~3 we give a proof of
the first step of the induction to prove our main result in Section~4.

\section{Preliminaries}
Throughout this paper we assume that $(X,r)$ is a finite colored space with $n=|X|$.
For all distinct $x,y\in X$ we write $r(x,y)$ instead of $r(\{x,y\})$ for short.
As mentioned in Section~1, $\mathsf{Im}(r)$ is identified with $A_2(r)$,
so that an element $[\{x,y\}]\in A_2(r)$ is denoted by $r(x,y)$.
For all distinct $x,y,z\in X$ we write $[\{x,y,z\}]\in A_3(r)$ as $\alpha\beta\gamma$
where $\alpha=r(x,y)$, $\beta=r(y,z)$ and $\gamma=r(z,x)$, so that
\[\alpha\beta\gamma=\beta\gamma\alpha=\gamma\alpha\beta=\alpha\gamma\beta=\gamma\beta\alpha=\beta\alpha\gamma.\]

For $\alpha\in A_2(r)$ we define
\[\mbox{$E_\alpha=r^{-1}(\alpha)$ and $R_\alpha=\{(x,y)\in X\times X\mid r(x,y)=\alpha\}$,}\]
so that, for each $\alpha\in A_2(r)$, $(X,E_\alpha)$ is a simple graph
and $R_\alpha$ is a symmetric binary relation on $X$.
For a binary relation $R$ on $X$ and $x\in X$ we set
\[R(x)=\{y\in X\mid (x,y)\in R\}.\]
For a finite colored space $(X_1,r_1)$ we say that $(X,r)$ is \textit{isomorphic} to $(X_1,r_1)$
if there exist bijections $f:X\to X_1$ and $g:\mathsf{Im}(r)\to \mathsf{Im}(r_1)$ with
\[g\circ r=r_1\circ f.\]
Remark that any isometric subspaces are isomorphic, but the converse does not hold in general.

For a positive integer $k$ with $1\leq k\leq n$ we set
\[M_k(r)=\{\alpha\in A_2(r)\mid \mbox{$(X,E_\alpha)$ has a vertex of degree at least $k$}\},\]
and we denote the size of $M_k(r)$ by $m_k(r)$.

\begin{rem}
For each $\alpha\in A_2(r)$,
\[\mbox{ $\alpha\notin M_2(r)$ if and only if
$(X,E_\alpha)$ is a matching on $X$,}\]
\[\mbox{and $\alpha\alpha\alpha\notin A_3(r)$ if and only if $(X,E_\alpha)$ is triangle-free.}\]
\end{rem}

For $\emptyset\ne \Gamma\subseteq A_2(r)$ we say that $\Gamma$ is \textit{closed}
if,
\[\mbox{for all $\alpha,\beta\in \Gamma$ and $\gamma\in A_2(r)$,
$\alpha\beta\gamma\in A_3(r)$ implies $\gamma\in  \Gamma$.}\]

\begin{lem}\label{lem:closed}
For $\emptyset\ne \Gamma\subseteq A_2(r)$, $\Gamma$ is closed if and only if
$(X,\bigcup_{\gamma\in \Gamma}E_\gamma)$ is a disjoint union of cliques.
\end{lem}
\begin{proof}
Set $R_0=\{(x,x)\mid x\in X\}$.
Then $\Gamma$ is closed if and only if $\bigcup_{\alpha\in \Gamma\cup\{0\}} R_\alpha$ is an equivalence relation on $X$.
Since the equivalence classes correspond to the connected components of
the graph $(X,\bigcup_{\gamma\in \Gamma}E_\gamma)$, the lemma holds.
\end{proof}

\begin{lem}\label{lem:m0}
If $M_{k+1}(r)=\emptyset$, then $n\leq 1+ka_2(r)$, and the equality does not hold when $k$ is odd and $a_2(r)$ is even.
\end{lem}
\begin{proof}
Let $x\in X$. Since $M_{k+1}(r)=\emptyset$, we have
\[n=|X|=1+\left|\bigcup_{\alpha\in A_2(r)}R_\alpha(x)\right|\leq 1+ka_2(r).\]
Suppose that the equality holds when $k$ is odd and $a_2$ is even.
Then each vertex has exactly $k$ neighbors in $(X,E_\alpha)$,
which implies that
\[2|E_\alpha|=|R_\alpha|=nk=(1+ka_2(r))k\]
 is odd, a contradiction.
\end{proof}

\begin{lem}\label{lem:m2}
If $a_2(r)=a_3(r)=m_2(r)$, then $a_2(r)\leq 2$.
\end{lem}
\begin{proof}
It is clear that, for all $\alpha,\beta,\gamma,\delta\in A_2(r)$,
if $\alpha\alpha\beta=\gamma\gamma\delta$, then $\alpha=\gamma$ and $\beta=\delta$.
This implies that $M_2(r)$ is embedded into $A_3(r)$.
Since $A_2(r)=M_2(r)$ by $a_2(r)=m_2(r)$,
it follows from $a_3(r)=m_2(r)$ that the mapping from $M_2(r)$ to $A_3(r)$
defined by
\[\alpha\mapsto \alpha\alpha\alpha'\]
for some $\alpha'\in A_2(r)$ is well-defined and bijective,
so that there are no tricolored triangles.

Suppose $a_2>1$.
Then $\alpha\alpha\beta\in A_3(r)$ for some $\alpha,\beta\in A_2(r)$ with $\alpha\ne \beta$.
Let $x_0,x_1,x_2,y\in X$ with
\[\mbox{$r(x_0,x_1)=r(x_0,x_2)=\alpha$, $r(x_1,x_2)=\beta$ and $y\notin \{x_0,x_1,x_2\}$.}\]
Since $\{x_0,x_1,y\}$ is not tricolored,
we have either
\[\mbox{ $r(x_0,y)=\gamma$ or $r(x_0,y)=\alpha$  where $\gamma:=r(x_1,y)$.}\]
If $r(x_0,y)=\gamma$, then $r(y,x_2)\ne \gamma$ by the injectivity of the above mapping,
and hence $r(y,x_2)=\alpha$, which implies that $\beta=\gamma$ or $\alpha=\gamma$ since
$\{x_1,x_2,y\}$ is not tricolored.
If $r(x_0,y)=\alpha$, then $r(y,x_2)=\gamma=\alpha$ by the injectivity.
Since $y$ is arbitrarily taken, we conclude that $r(x_1,y)\in \{\alpha,\beta\}$.
Since there are no tricolored triangles, it follows that $A_2(r)=\{\alpha,\beta\}$, and hence $a_2(r)=2$.
\end{proof}

\begin{lem}\label{lem:delta}
For all $\alpha,\beta,\gamma,\delta\in A_2(r)$, if $\beta\gamma\delta\in A_3(r)$ is a unique element in which $\delta$ appears,
and $\alpha\notin \{\beta,\gamma,\delta\}$,
then either $\alpha\beta\gamma\in A_3(r)$ or $\beta\beta\alpha,\gamma\gamma\alpha\in A_3(r)$, and furthermore, $\beta,\gamma\in M_2(r)$ unless $n\leq 4$.
\end{lem}
\begin{proof}
Let $(x,y)\in R_\delta$ and $z\in X\setminus \{x,y\}$.
Then, by the assumption, $[\{x,y,z\}]=\beta\gamma\delta$, and hence,
\begin{equation}\label{eq:1}
(r(x,z),r(y,z))\in \{(\beta,\gamma),(\gamma,\beta)\}.
\end{equation}
For all distinct $z,w\in X$ with $r(z,w)=\alpha$,
we have $\{z,w\}\cap \{x,y\}=\emptyset$ by the assumption of $\alpha\notin \{\beta,\gamma,\delta\}$.
Thus, $([\{x,z,w\}],[\{y,z,w\}])$ is either
\[\mbox{$(\alpha\beta\gamma,\alpha\beta\gamma)$ or
$(\alpha\beta\beta,\alpha\gamma\gamma)$}.\]
Moreover, if $n\geq 5$, then $\beta,\gamma\in M_2(r)$ by (\ref{eq:1}) and $|X\setminus\{x,y\}|\geq 3$.
\end{proof}

\section{Proof of the first step of the induction}
Throughout this section we assume that $(X,r)$ is a finite colored space with $a_2(r)=a_3(r)=4$.

\begin{lem}\label{lem:3c}
For all distinct $\alpha,\beta,\gamma\in A_2(r)$, if $\{\alpha,\beta,\gamma\}$ is closed,
then
\[\alpha\delta\delta, \beta\delta\delta,\gamma\delta\delta\in A_3(r)\]
 where $\delta$ is
a unique element in $A_2(r)\setminus\{\alpha,\beta,\gamma\}$.
\end{lem}
\begin{proof}
By Lemma~\ref{lem:closed},
$(X,E_\alpha\cup E_\beta\cup E_\gamma)$ is a disjoint union of cliques.
Since $(X,E_\delta)$ is its complement,
each element as in the conclusion can be realized as a triangle in this configuration.
\end{proof}

\begin{lem}\label{lem:aab}
Suppose that $\alpha,\beta\in A_2(r)$ satisfy the following:
\begin{enumerate}
\item $\{\alpha,\beta\}$ is closed and $\alpha\ne \beta$;
\item We have $\alpha\alpha\beta\in A_3(r)$;
\item There exists an element in $A_3(r)\setminus\{\alpha\alpha\beta\}$ which forms
$\alpha_1\alpha_2\alpha_3$ for some $\alpha_1,\alpha_2, \alpha_3\in\{\alpha,\beta\}$.
\end{enumerate}
Let $Y$ be a connected component of $(X,E_\alpha\cup E_\beta)$ with
$x_0,x_1,x_2\in Y$,
\[\mbox{$r(x_0,x_1)=r(x_0,x_2)=\alpha$, $r(x_1,x_2)=\beta$, and $x\in X\setminus Y$.}\]
Then $r(x_1,x)=r(x_2,x)\ne r(x_0,x)\notin M_2(r)$. In particular, $|X\setminus Y|=1$ and
\[A_3(r)=\{\beta\beta\beta, \alpha\alpha\beta,\gamma\gamma\beta,\alpha\gamma\delta\}\]
where $\gamma=r(x_2,x)$ and $\delta=r(x_0,x)$.
\end{lem}
\begin{proof}
We claim that $r(x_i,x)\ne r(x_0,x)$ for $i=1,2$.
Suppose the contrary, i.e., $r(x_i,x)= r(x_0,x)$ for some $i=1,2$.
Without loss of generality we may assume that $r(x_1,x)= r(x_0,x)$.
Notice that
\[\delta\delta\alpha, \delta\gamma\alpha, \delta\gamma\beta\in A_3(r),\]
and they are distinct if $\gamma\ne \delta$, which contradicts $a_3(r)=4$ by (ii) and (iii).
On the other hand, if $\gamma=\delta$, then $\delta\delta\alpha,\delta\delta\beta\in A_3(r)$ are distinct,
a contradiction to $a_2(r)=4$ by (ii) and (iii).

Since $r(x_i,x)\in A_2(r)\setminus\{\alpha,\beta\}$ by (i), it follows from the above claim that
\[\mbox{$r(x_1,x)=r(x_2,x)\ne r(x_0,x)$, and $\gamma\gamma\beta, \gamma\delta\alpha\in A_3(r)$ are distinct.}\]
Therefore, we conclude from (ii) and (iii) that $\delta\notin M_2(r)$.
Since $x$ is arbitrarily taken in $X\setminus Y$, it follows from $\delta\notin M_2(r)$ that
\[|X\setminus Y|=1.\]
This implies that
\[\alpha\alpha\alpha,\alpha\beta\beta\notin A_3(Y),\]
 otherwise,
\[\mbox{$\alpha\gamma\gamma\in A_3(r)$ or
$\beta\gamma\delta\in A_3(r)$,}\]
a contradiction to $a_3(r)=4$.
Since $\gamma\gamma\beta, \gamma\delta\alpha\in A_3(r)$ are distinct,
the last statement follows from (ii) and (iii) with $a_2(r)=a_3(r)=4$.
\end{proof}

\begin{lem}\label{lem:n8}
Suppose $m_3(r)>0$ and
there is no $\alpha\in A_2(r)$ such that $\alpha\alpha\alpha\in A_3(r)$.
Then, for a suitable ordering of elements of $A_2(r)$, $A_3(r)$ equals one of the following:
\begin{enumerate}

\item $\{\alpha\alpha\beta,\alpha\alpha\gamma,\alpha\alpha\delta,\beta\gamma\delta\}$ and $n\leq 8$;
\item $\{\alpha\alpha\beta,\alpha\alpha\gamma, \beta\beta\gamma, \alpha\alpha\delta\}$ and $n\leq 6$;

\item $\{\alpha\alpha\beta,\alpha\alpha\gamma, \beta\beta\gamma, \alpha\beta\delta\}$ and $n\leq 6$.
\end{enumerate}
\end{lem}
\begin{proof}
Let $\alpha\in M_3(r)$ and $\{x,x_1,x_2,x_3\}\in \binom{X}{4}$ such that
\[\mbox{$r(x,x_i)=\alpha$ for $i=1,2,3$.}\]
For short we denote $\{x_1,x_2,x_3\}$ by $U$.
Notice that, by the assumption,
\[\mbox{$\alpha\notin A_2(r_U)$ and $1<|A_2(r_U)|\leq 3$}.\]
If $|A_2(r_U)|=3$, then we obtain the first case
\[A_3(r)=\{\alpha\alpha\beta,\alpha\alpha\gamma,\alpha\alpha\delta,\beta\gamma\delta\}\]
where $A_3(r)\setminus\{\alpha\}=\{\beta,\gamma,\delta\}$. Since $\{\beta,\gamma,\delta\}$ is closed
and $\beta\gamma\delta$ is a unique element consisting of some of $\{\beta,\gamma,\delta\}$,
it follows from \cite[Thm.~3.1]{hs} and $\alpha\alpha\alpha\notin A_3(r)$ that
$(X,E_\beta\cup E_\gamma\cup E_\delta)$ has exactly two connected components of size at most four.
Therefore, $n\leq 8$.

Suppose $A_2(r_U)=\{\beta,\gamma\}$ where $\alpha, \beta,\gamma$ are distinct.
Without loss of generality we may assume that
\[r(x_1,x_2)=r(x_2,x_3)=\beta, r(x_1,x_3)=\gamma.\]
Note that
\[\mbox{$\alpha\alpha\beta,\alpha\alpha\gamma, \beta\beta\gamma\in A_3(r)$ are distinct,}\]
 and
the unique element $\delta \in A_2(r)\setminus\{\alpha,\beta,\gamma\}$ appears in
the unique element in $A_3(r)\setminus\{ \alpha\alpha\beta,\alpha\alpha\gamma, \beta\beta\gamma\}$.
We shall write the unique element in $A_3(r)$ as $\mu\nu \delta$ where $\mu,\nu\in \{\alpha,\beta,\gamma,\delta\}$.

We claim that $\alpha \in \{\mu,\nu\}$.
Otherwise, $\{\beta,\gamma,\delta\}$ is closed, by Lemma~\ref{lem:3c},
\[\alpha\alpha\delta\in A_3(r),\]
a contradiction.

Without loss of generality we may assume that
$\mu=\alpha$ and $\nu\in \{\alpha, \beta,\gamma,\delta\}$.
Then $\{\beta,\gamma\}$ is closed for each case.
Since $\beta\beta\gamma$ is a unique element of $A_3(r)$ consisting of $\beta$ and $\gamma$,
it follows from \cite[Thm.~3.1]{hs} that each connected component of $(X,E_\beta\cup E_\gamma)$ has
size at most four.

If $\nu=\alpha$, then we obtain the second case. Since $\{\beta,\gamma,\delta\}$ is closed, it follows that
$(X,E_\beta\cup E_\gamma\cup E_\delta)$ has exactly two connected components, one of which
contains $\beta\beta\gamma$, another of which consists of two vertices with color $\delta$. Therefore, $n\leq 6$.

From now on we assume that $\nu\ne \alpha$.
We claim $(X,E_\beta\cup E_\gamma)$ has exactly two connected components.
Clearly, it has more than one connected component.
If $y$ is an element of $X$ belonging to neither of the connected components
containing $x$ nor $x_1$, then
\[\mbox{$r(y,x_i)=\alpha$ for some $i=1,2,3$,}\]
 otherwise, $\delta\delta\beta\in A_3(r)$, a contradiction.
Since $\alpha\alpha\alpha\notin A_3(r)$, it follows that
$r(x,y)=\delta$, which implies that $\alpha\alpha\delta\in A_3(r)$, a contradiction to $\nu\ne \alpha$.

Since $\alpha\nu\delta\in A_3(r)$, it follows from the claim that
one edge with color $\delta$ lies between the two connected components and
$\nu\in \{\beta,\gamma\}$.
Notice that $\alpha\beta\delta$ and $\alpha\gamma\delta$ are distinct so that
only one of them belonged to $A_3(r)$.
This implies that each connected component has size at most three and
$|R_\delta(x_2)|=1$. Therefore, $\nu=\beta$ and (iii) holds.

This completes the proof.
\end{proof}

\begin{ex}
Let $\{Y,Z\}$ be a bipartition of $X$ with $|Y|=|Z|=4$.
We set $E_\alpha=(Y,Z)$ where $(Y,Z)$ is defined to be 
the set of edges between $Y$ and $Z$,
and $E_\beta$, $E_\gamma$, $E_\delta$ to be three disjoint perfect matching disjoint from $E_\alpha$.
Then each subset $W$ of $X$ with $|W|\geq 4$ and $|W\cap Y|\ne 2$ satisfies
\[A_3(W)=\{\alpha\alpha\beta,\alpha\alpha\gamma,\alpha\alpha\delta,\beta\gamma\delta\}.\]
\end{ex}

\begin{ex}
Let $\{Y,Z\}$ be a bipartition of $X$ with $|Y|=4$ and $|Z|=2$
We set
\[
\mbox{$E_\alpha=(Y,Z)$, and
$E_\gamma$ to be a perfect matching on $Y$,}\]
\[\mbox{ $E_\beta=\binom{Y}{2}\setminus E_\gamma$,
$E_\delta=\binom{Z}{2}$.}\]
Then each subset $W$ of $X$ with $|W|\geq 5$ and $Z\subseteq W$ satisfies
\[A_3(W)=\{\alpha\alpha\beta,\alpha\alpha\gamma, \beta\beta\gamma, \alpha\alpha\delta\}.\]
\end{ex}

\begin{ex}
Let $\{Y,Z\}$ be a bipartition of $X$ with $|Y|=|Z|=3$ and $(y_0,z_0)\in Y\times Z$.
We set
\[\mbox{$E_\alpha=(Y,Z)\setminus\{y_0z_0\}$, $E_\delta=\{y_0z_0\}$, $
E_\gamma=\binom{Y\setminus\{y_0\}}{2}\cup \binom{Z\setminus\{z_0\}}{2}$}, \]
and $E_\beta$ to be the remaining.
Then each subset $W$ of $X$ with $4\leq |W|$, $\{y_0,z_0\}\subseteq W$ and $|Y\cap W|\ne 2$ satisfies
\[A_3(W)=\{\alpha\alpha\beta,\alpha\alpha\gamma, \beta\beta\gamma, \alpha\beta\delta\}.\]
\end{ex}

\begin{lem}\label{lem:n9}
Suppose that
there exists $\alpha\in A_2(r)$ such that
\[\mbox{$\alpha\alpha\alpha\in A_3(r)$ and $\{\alpha\}$ is closed.}\]
Then, for a suitable ordering of elements of $A_2(r)$, $A_3(r)$ equals one of the following:
\begin{enumerate}
\item $\{\alpha\alpha\alpha, \alpha\beta\gamma,\alpha\gamma\delta,\alpha\beta\delta\}$ and $n\leq 6$;
\item $\{\alpha\alpha\alpha,\alpha\beta\beta, \gamma\gamma\alpha, \beta\gamma\delta\}$;
\item $\{\alpha\alpha\alpha, \alpha\beta\beta,\alpha\beta\gamma,\alpha\beta\delta\}$.
\end{enumerate}
\end{lem}

\begin{proof}
Let $Y$ be a connected component of $(X,E_\alpha)$ containing distinct elements $x_1,x_2,x_3$, and let $x\in X\setminus Y$.
Then
\[\alpha\notin \{r(x,x_i)\mid i=1,2,3\}.\]
If $|\{r(x,x_i)\mid i=1,2,3\}|=3$, then we obtain the first case
\[A_3(r)=\{\alpha\alpha\alpha, \alpha\beta\gamma,\alpha\gamma\delta,\alpha\delta\beta\}\]
where $\{r(x,x_i)\mid i=1,2,3\}=\{\beta,\gamma,\delta\}$.
Since $\beta,\gamma,\delta\notin M_2(r)$,
\[|Y|=|R_\beta(x)|+|R_\gamma(x)|+|R_\delta(x)|\leq 3.\]
This implies that $n\leq 6$.

If $|\{r(x,x_i)\mid i=1,2,3\}|<3$, then there exists $\beta\in A_2(r)$ such that $\alpha\beta\beta\in A_3(r)$.

Suppose that $\{\alpha,\beta\}$ is closed.
Applying Lemma~\ref{lem:aab} for $\alpha\beta\beta\in A_3(r)$ we obtain the second case
\[A_3(r)=\{\alpha\alpha\alpha,\alpha\beta\beta,\gamma\gamma\alpha,\beta\gamma\delta\}\]
where $A_2(r)\setminus\{\alpha,\beta\}=\{\gamma,\delta\}$.

Suppose that $\{\alpha,\beta\}$ is not closed.
Note that there are no $\gamma \in A_2(r)$ such that
\[\mbox{$\alpha\alpha \gamma\in A_3(r)$}\]
 since
$\{\alpha\}$ is closed.
Thus there exists $\gamma\in A_2(r)\setminus\{\alpha,\beta\}$ such that one of the following holds:
\begin{enumerate}
\item $\alpha\beta\gamma\in A_3(r)$;
\item $\beta\beta\gamma\in A_3(r)$.
\end{enumerate}

Notice that, for each case, $\{\alpha,\beta,\gamma\}$ is not closed by Lemma~\ref{lem:3c}.
This implies $A_3(r)$ contains $\delta\mu\nu$ such that $\mu,\nu\in \{\alpha,\beta,\gamma\}$
where $\{\mu,\nu\}$ is one of the following:
\[\{\alpha,\beta\},\{\alpha,\gamma\},\{\beta\},\{\beta,\gamma\}, \{\gamma\}.\]
Applying Lemma~\ref{lem:delta} for a unique element $\mu\nu\delta\in A_3(r)$ including $\delta$
we can eliminate the case of $\gamma\gamma\delta \in A_3(r)$, and hence, $\gamma\notin M_2(r)$.
Since $n\geq 5$ unless the first case of this lemma holds,
it follows from Lemma~\ref{lem:delta} that
$\gamma\notin \{\mu,\nu\}$, and the following case can be eliminated:
\begin{enumerate}
\item $\alpha\beta\gamma, \alpha\beta\delta \in A_3(r)$, which is the third case;

\item $\alpha\beta\gamma, \beta\beta\delta \in A_3(r)$ does not occur since $\beta\beta\gamma\notin A_3(r)$;

\item $\beta\beta\gamma, \alpha\beta\delta\in A_3(r)$ does not occur since $\alpha\beta\gamma,\alpha\alpha\gamma\notin A_3(r)$;

\item $\beta\beta\gamma, \beta\beta \delta\in A_3(r)$ does not occur since $\beta\beta\beta\notin A_3(r)$ and $\{\alpha,\delta,\gamma\}$ is closed.
\end{enumerate}
This completes the proof.
\end{proof}

\begin{ex}
Let $\{Y,Z\}$ a bipartition of $X$ with $|Y|=|Z|=3$. We set
\[\mbox{ $E_\alpha=\binom{Y}{2}\cup\binom{Z}{2}$,}\]
$E_\beta$, $E_\gamma$ and $E_\delta$ to be disjoint perfect matchings on $X$ connecting $Y$ and $Z$.
Then each subset $W$ of $X$ with $|W|\geq 4$ and $|Y\cap W|\ne 2$ satisfies
$A_3(W)=\{\alpha\alpha\alpha, \alpha\beta\gamma,\alpha\gamma\delta,\alpha\delta\beta\}$.
\end{ex}

\begin{lem}\label{lem:n10}
Suppose that
there exists $\alpha\in A_2(r)$ such that $\alpha\alpha\alpha\in A_3(r)$ and $\{\alpha\}$ is not closed.
Then, for a suitable ordering of elements of $A_2(r)$,
\[A_3(r)=\{\alpha\alpha\alpha, \alpha\alpha\beta,\alpha\alpha\gamma,\alpha\alpha\delta\}.\]
\end{lem}
\begin{proof}
Since $\{\alpha\}$ is not closed, there exists $\beta\in A_2(r)\setminus\{\alpha\}$ such that
\[\alpha\alpha\beta\in A_3(r).\]

Applying Lemma~\ref{lem:aab} with the assumption of
 \[\alpha\alpha\alpha,\alpha\alpha\beta\in A_3(r)\]
we obtain that $\{\alpha,\beta\}$ is not closed.
Thus, there exists $\gamma\in A_2(r)\setminus\{\alpha,\beta\}$ such that one of the following holds:
\begin{enumerate}
\item $\alpha\alpha\gamma\in A_3(r)$;
\item $\alpha\beta\gamma\in A_3(r)$;
\item  $\beta\beta\gamma\in A_3(r)$.
\end{enumerate}

Notice that, for each case, $\{\alpha,\beta,\gamma\}$ is not closed,
otherwise, we have $a_3(r)>4$, a contradiction.
This implies that $A_3(r)$ contains $\delta\mu\nu$ such that $\mu,\nu\in \{\alpha,\beta,\gamma\}$
where $\{\mu,\nu\}$ is one of the following:
\[\{\alpha\}, \{\alpha,\beta\},\{\alpha,\gamma\},\{\beta\},\{\beta,\gamma\},\{\gamma\}.\]
Applying Lemma~\ref{lem:delta} we can show the following cases never occur:
\[\mbox{$\gamma\gamma\delta\in A_3(r)$, $\beta\beta\delta\in A_3(r)$ or 
$\beta\beta\gamma\in A_3(r)$,}\]
and hence, $\gamma\notin M_2(r)$ and $\beta\notin M_2(r)$.
Since $n\geq 5$ by $a_2(r)=a_3(r)=4$ and the assumption on $\alpha$,
it follows from Lemma~\ref{lem:delta} that
$\gamma\notin \{\mu,\nu\}$, and the following cases can be eliminated:
\begin{enumerate}

\item $\alpha\alpha\gamma,\alpha\alpha\delta\in A_3(r)$, which is the same as in the statement;
\item $\alpha\alpha\gamma,\alpha\beta\delta\in A_3(r)$ does not occur since $\alpha\beta\gamma, \beta\beta\gamma\notin A_3(r)$;

\item $\alpha\beta\gamma, \alpha\alpha\delta \in A_3(r)$ does not occur since $\alpha\alpha\gamma\notin A_3(r)$;

\item $\alpha\beta\gamma, \alpha\beta\delta \in A_3(r)$ does not occur since $\beta\notin M_2(r)$;

\end{enumerate}
This completes the proof.
\end{proof}

\begin{thm}\label{thm:31}
The statement of Theorem ~\ref{thm:1} holds when $a_2(r)=a_3(r)=4$.
\end{thm}
\begin{proof}

Suppose $m_3(r)=0$. 
Then $n=9$ since $a_2(r)=4$ and $9\leq n$.
This implies $m_2(r)=4$, which contradicts Lemma~\ref{lem:m2}.

Suppose $m_3(r)>0$.
Applying Lemma~\ref{lem:n8} with $9\leq n$ we conclude that
there exists $\alpha\in A_2(r)$ such that $\alpha\alpha\alpha\in A_3(r)$.
By Lemma~\ref{lem:n9} and \ref{lem:n10}, $A_3(r)$ is one of the following:
 \begin{enumerate}
\item $\{\alpha\alpha\alpha, \alpha\beta\beta,\alpha\gamma\gamma,\beta\gamma\delta\}$;
\item $\{\alpha\alpha\alpha, \alpha\beta\beta,\alpha\beta\gamma,\alpha\beta\delta\}$;
\item $\{\alpha\alpha\alpha, \alpha\alpha\beta,\alpha\alpha\gamma,\alpha\alpha\delta\}$.
\end{enumerate}

(i) Let $z_1,z_2\in X$ with $r(z_1,z_2)=\delta$ and $Y$ a connected component of $(X,E_\alpha)$ with size at least three.
Since no element of $A_3(r)$ contains both $\alpha$ and $\delta$, we have
$\{z_1,z_2\}\cap Y=\emptyset$.
Since $\delta\notin M_2(r)$,
\[\mbox{$r(y,z_i)\in\{\beta,\gamma\}$ for each $y\in Y$.}\]
Since $\alpha\beta\gamma\notin A_3(r)$,
\[\mbox{$r(y,z_i)$ is constant whenever $y\in Y$.}\]
Since $\beta$ and $\gamma$ are symmetric, we may assume that
$r(y,z_1)=\beta$ for each $y\in Y$.
Since $\beta\beta\delta\notin A_3(r)$, it follows that
$r(y,z_2)=\gamma$ for each $y\in Y$. Thus, the partition induces the one given in (iii) of Theorem~\ref{thm:1}.

(ii) Since each element of $A_3(r)$ contains $\alpha$, it follows from Lemma~\ref{lem:closed} that
$(X,E_\alpha)$ has exactly two connected components, say $Y$ and $Z$.
Note that $E_\gamma$ and $E_\delta$ are matchings on $X$ and $E_\beta\cup E_\gamma$ is also a matching on $X$ since
no element of $A_3(r)$ contains both $\beta$ and $\gamma$. Therefore, $(X,r)$ is the one given in (ii) of Theorem~\ref{thm:1}.

(iii) Note that that $E_\beta$, $E_\gamma$ and $E_\delta$ are matchings on $X$ and
$E_\beta\cup E_\gamma\cup E_\delta$ is also a matching on $X$ since
no element of $A_3(r)$ contains two of $\beta$, $\gamma$ and $\delta$.
Therefore, $(X,r)$ coincides with the one given in  (i) of Theorem~\ref{thm:1}.

\end{proof}

\section{Proof of our main results}
For functions $r,r_1$ whose domain is $\binom{X}{2}$ we say that
$r_1$ is a \textit{fusion} of $r$ if $\{r^{-1}(\alpha)\mid \alpha\in \mathsf{Im}(r)\}$ is a refinement of
$\{r_1^{-1}(\alpha)\mid \alpha\in \mathsf{Im}(r_1)\}$.
\begin{lem}\label{lem:a1}
Let $(X,r)$ be a finite colored space with
\[\mbox{$2\leq a_2(r)$,  $0<m_2(r)$ and $5\leq n$.}\]
Then there exists a fusion $r_1$ of $r$ such that
\[\mbox{ $a_2(r)-a_2(r_1)=1$ and $1\leq a_3(r)-a_3(r_1)$.}\]
\end{lem}
\begin{proof}
Let $\{x_1,x_2,x_3,x_4\}\in \binom{X}{4}$ with $r(x_1,x_2)=r(x_2,x_3)$.

If $r(x_4,x_1)\ne r(x_4,x_3)$, then the identification of  $r(x_4,x_1)$ and $r(x_4,x_3)$ is a required one since
\[\mbox{$[\{x_1,x_2,x_4\}]= [\{x_3,x_2,x_4\}]$ in $(X,r_1)$ but not in $(X,r)$.}\]
If $r(x_4,x_1)= r(x_4,x_3)\ne r(x_1,x_2)$, then the identification of  $r(x_1,x_2)$ and $r(x_1,x_4)$ is a required one since
\[\mbox{$[\{x_1,x_2,x_3\}]= [\{x_1,x_4,x_3\}]$ in $(X,r_1)$ but not in $(X,r)$.}\]
If $r(x_4,x_1)= r(x_4,x_3)= r(x_1,x_2)$ and $r(x_1,x_3)\ne r(x_2,x_4)$, then
the identification of  $r(x_1,x_3)$ and $r(x_2,x_4)$ is a required one since
\[\mbox{$[\{x_1,x_2,x_3\}]= [\{x_2,x_1,x_4\}]$ in $(X,r_1)$ but not in $(X,r)$.}\]
Since $x_4$ is arbitrarily taken, we may assume that, for every $x_5\in X\setminus \{x_i\mid i=1,2,3,4\}$, we have $r(x_5,x_1)= r(x_5,x_3)= r(x_1,x_2)$ and
\[r(x_2,x_4)=r(x_1,x_3)=r(x_2,x_5)=r(x_4,x_5).\]
If $r(x_1,x_2)\ne r(x_2,x_4)$, then the identification of  $r(x_1,x_2)$ and $r(x_2,x_4)$ is a required one since
\[\mbox{$[\{x_1,x_2,x_3\}]= [\{x_2,x_4,x_5\}]$ in $(X,r_1)$ but not in $(X,r)$.}\]
If $r(x_1,x_2)=r(x_2,x_4)$, then $a_2(r)=1$, a contradiction.
\end{proof}

\begin{lem}\label{lem:a2}
Let $(X,r)$ be a finite colored space with
\[\mbox{ $a_2(r)<\binom{n}{2}$, $m_2(r)=0$ and $5\leq n$.}\]
Then there exists a fusion $r_1$ of $r$ such that
\[\mbox{$a_2(r)-a_2(r_1)\leq 2$, and $a_2(r)-a_2(r_1)\leq a_3(r)-a_3(r_1)$.}\]
\end{lem}
\begin{proof}
The condition $m_2(r)=0$ implies that, for each $\alpha\in \mathsf{Im}(r)$,
$(X,E_\alpha)$ is a matching.
The condition $a_2(r)<\binom{n}{2}$ implies that
there exists $\alpha\in\mathsf{Im}(r)$ such that $|r^{-1}(\alpha)|\geq 2$.
Let
\[\mbox{$x_1x_2, y_1y_2\in r^{-1}(\alpha)$ with  $x_1x_2\cap y_1y_2=\emptyset$,}\]
so that, for $i=1,2$,
\[\mbox{$r(x_i,y_1)\ne r(x_i,y_2)$ and $r(x_1,y_i)\ne r(x_2,y_i)$}\]
where we write $\{u,v\}\in \binom{X}{2}$ as $uv$ for short.

If $|\{r(x_i,y_j)\mid i,j=1,2\}|=4$, then the identification of  $r(x_1,y_1)$ and $r(x_2,y_2)$ is a required one since
\[\mbox{$[\{x_1,x_2,y_1\}]= [\{y_1,y_2,x_2\}]$ in $(X,r_1)$ but not in $(X,r)$.}\]

If $|\{r(x_i,y_j)\mid i,j=1,2\}|=3$ so that we may assume $r(x_1,y_1)=r(x_2,y_2)$,
then the identification of  $r(x_1,y_2)$ and $r(x_2,y_1)$ is a required one since
\[\mbox{$[\{x_1,x_2,y_1\}]= [\{x_1,x_2,y_2\}]$ in $(X,r_1)$ but not in $(X,r)$.}\]

If $|\{r(x_i,y_j)\mid i,j=1,2\}|=2$ so that we may assume
\[\mbox{$r(x_1,y_1)=r(x_2,y_2)$ and $r(x_1,y_2)=r(x_2,y_1)$, }\]
then we take $z\in X\setminus \{x_1,x_2,y_1,y_2\}$, so that
\[\mbox{$|\{r(z,x_i),r(z,y_i)\mid i=1,2\}|=4$ and}\]
\[\mbox{ $\{r(z,x_i),r(z,y_i)\mid i=1,2\}\cap \{r(x_1,x_2), r(x_1,y_1),r(x_1,y_2)\}=\emptyset$.}\]
Furthermore, the fusion $r_1$ of $r$ obtained by identifying $r(z,x_1)$ with $r(z,y_2)$, and  $r(z,x_2)$ with $r(z,y_1)$, is a required one,
since
\[\mbox{$[\{x_1,x_2,z\}]= [\{y_1,y_2,z\}]$ and $[\{x_1,y_1,z\}]= [\{x_2,y_2,z\}]$ in $(X,r_1)$}\]
 but not in $(X,r)$.

Since $|\{r(x_i,y_j)\mid i,j=1,2\}|>1$, this completes the proof.
\end{proof}

\begin{thm}\label{thm:a3}
For each finite colored space $(X,r)$ with $5\leq n$ we have $a_2(r)\leq a_3(r)$.
\end{thm}
\begin{proof}
Suppose the contrary, i.e., $a_2(r)> a_3(r)$.
If $a_2(r)=\binom{n}{2}$, then
\[\mbox{$a_3(r)=\binom{n}{3}$, so that
$a_2(r)\leq a_3(r)$ unless $n\leq 4$.}\]
Applying Lemma~\ref{lem:a1} and \ref{lem:a2} for $(X,r)$ we obtain a fusion $r_1$ of $r$
such that
\[\mbox{$a_2(r_1)= a_2(r)-1>a_3(r)-1\geq a_3(r_1)$ if $a_2(r)-a_2(r_1)=1$}, \] and
\[\mbox{$a_2(r_1)=a_2(r)-2>a_3(r)-2\geq a_3(r_1)$ if $a_2(r)-a_2(r_1)=2$}\]
Repeating this argument we have, for some positive integer $k$,
\[2=a_2(r_k)>a_3(r_k) \mbox{ or }  3=2_3(r_k)>a_3(r_k),\]
which contradicts \cite[Thm.~3.5]{hs}.
\end{proof}
\begin{flushleft}
Proof of Theorem~\ref{thm:1}.
\end{flushleft}
\begin{proof}
Use induction on $a_2(r)$.
Theorem~\ref{thm:31} proves the first step of the induction.
Suppose $5\leq a_2(r)$.

We claim that $m_2(r)>0$.
Suppose the contrary, i.e., $(X,E_\alpha)$ is a matching for each $\alpha\in A_2(r)$.
Then
\[\binom{n}{2}=\sum_{\alpha\in \mathsf{Im}(r)}|r^{-1}(\alpha)|\leq a_2(r)\frac{n}{2},\]
and hence, $8\leq n-1\leq a_2(r)$.
Therefore, by Lemma~\ref{lem:a1} and \ref{lem:a2} and Theorem~\ref{thm:a3}, there exists a fusion $r_1$ of $r$ such that
\[6\leq a_2(r_1)=a_3(r_1)<a_2(r).\]
By the inductive hypothesis, $(X,r_1)$ induces a partition given in Theorem~\ref{thm:1}.
However, it is impossible to obtain $(X,r)$ with $m_2(r)=0$ by separating one or two elements of $A_2(r_1)$ into two parts, a contradiction.

By the claim, we conclude from Lemma~\ref{lem:a1} that
$(X,r)$ is obtained by separating only one of $r_1^{-1}(\alpha)$ with
 $\alpha\in \mathsf{Im}(r_1)$ into two parts.

First, we claim that $(X,r_1)$ does not induce a partition given in Theorem~\ref{thm:1}(iii).
Suppose the contrary, i.e., $(X,r_1)$ induces a partition given in Theorem~\ref{thm:1}(iii).
For convenience we assume that
\[\mbox{$\mathsf{Im}(r_1)=\{\alpha,\beta,\gamma,\delta\}$, $r_1^{-1}(\delta)=\{y,z\}$, }\]
\[\mbox{$r_1^{-1}(\gamma)$ is the edges from $y$ to $U$ where $U:=X\setminus\{y,z\}$,}\]
\[\mbox{
$r_1^{-1}(\beta)$ is the edges from $z$ to $U$ and
$r_1^{-1}(\alpha)=\binom{U}{2}$.}\]
Notice that $r_1^{-1}(\delta)$ is not separated since it is only one edge.
If $r_1^{-1}(\gamma)$ is separated into two parts, say $r^{-1}(\gamma_1)$  and $r^{-1}(\gamma_2)$ then
\[a_3(r)-a_3(r_1)\geq 2\]
 since
\[\mbox{$\gamma_1\gamma_2\alpha$, $\gamma_1\beta\delta$
 and $\gamma_2\beta\delta$}\]
  are non-isometric in $(X,r)$.
This implies
\begin{equation}\label{eq:b1}
a_3(r)\geq a_3(r_1)+2=6>5=1+a_2(r_1)=a_2(r),
\end{equation}
a contradiction to $a_2(r)=a_3(r)$.
By symmetry of $\beta$ and $\gamma$, we can show that $r_1^{-1}(\beta)$ is not separated.
If $r_1^{-1}(\alpha)$ is separated into two parts, say $r^{-1}(\alpha_1)$  and $r^{-1}(\alpha_2)$, then
$a_3(r)-a_3(r_1)\geq 2$ since
$\alpha_1\beta\gamma$ and $\alpha_2\beta\gamma$ are not isometric in $(X,r)$,
and $a_3(r_U)\geq 2$ since $|U|\geq 5$ and $a_2(r_U)=2$.
Thus, we have the same contradiction as (\ref{eq:b1}).

Second, we claim that, if  $(X,r_1)$ induces a partition given in Theorem~\ref{thm:1}(ii),
then $(X,r_1)$ induces a partition given in  Theorem~\ref{thm:1}(ii).
For convenience we assume that
\[\mathsf{Im}(r_1)=\{\alpha,\beta,\gamma_i\mid i=1,2,\ldots, a_2(r)-3\},\]
\[\mbox{$r_1^{-1}(\alpha)=\binom{Y}{2}\cup \binom{Z}{2}$,
$\bigcup_{i=1}^{a_2(r)-3}r^{-1}(\gamma_i)$ is a matching, }\]
\[\mbox{and
$r_1^{-1}(\beta)$ is the remaining edges.}\]
If $r_1^{-1}(\gamma_i)$ is separated, then $(X,r)$ induces a partition
 given in Theorem~\ref{thm:1}(ii).
If $r_1^{-1}(\beta)$ is separated into two parts, say $r^{-1}(\beta_1)$ and $r^{-1}(\beta_2)$,
then one of them should be a matching disjoint from
\[\bigcup_{i=1}^{a_2(r)-3}r^{-1}(\gamma_i),\]
otherwise, the following non-isometric elements in $(X,r)$ would induce a contradiction:
\[\mbox{$\alpha\beta_1\beta_1$, $\alpha\beta_2\beta_2$,
 and $\alpha\beta_1\beta_2$ or}\]
\[\mbox{$\alpha\beta_i\beta_i$, $\alpha\beta_i\beta_j$, $\alpha\beta_i\gamma_k$,
 and $\alpha\beta_j\gamma_l$ for some $i,j,k,l$ with $i\ne j$.}\]
If $r_1^{-1}(\alpha)$ is separated into two parts, say $r^{-1}(\alpha_1)$ and $r^{-1}(\alpha_2)$,
then we have a contradiction because of the following non-isometric elements:
\[\mbox{$\alpha_1\beta_i\beta_j$, $\alpha_2\beta_i\beta_j$ for some $i,j$}\]
\[\mbox{and $|A_3(r_Y)\cup A_3(r_Z)|\geq 2$ since $\max\{|Y|,|Z|\}\geq 5$}.\]

It remains to eliminate the case where $(X,r_1)$ induces
a partition given in Theorem~\ref{thm:1}(i),
and we shall prove that $(X,r)$ induces a partition given in Theorem~\ref{thm:1}(i),(ii).
For convenience we assume that
\[\mathsf{Im}(r_1)=\{\alpha,\gamma_i\mid i=1,2,\ldots, a_2(r)-2\},\]
\[\mbox{$\bigcup_{i=1}^{a_2(r)-2}r^{-1}(\gamma_i)$ is a matching, and}\]
\[\mbox{$r_1^{-1}(\alpha)$ is the remaining edges.}\]
If $r_1^{-1}(\gamma_i)$ is separated, then $(X,r)$ induces
a partition given in Theorem~\ref{thm:1}(i).
Suppose $r_1^{-1}(\alpha)$ is separated into two parts, say $r^{-1}(\alpha_1)$ and $r^{-1}(\alpha_2)$.
If  $\alpha_i\notin M_2(r)$, then
\[\left(\bigcup_{i=1}^{a_2(r)-2}r^{-1}(\gamma_i)\right)\cup r^{-1}(\alpha_i)\]
 is a matching,
otherwise,  the following non-isometric elements would induce a contradiction:
\[\mbox{$\alpha_i\alpha_j\alpha_j$, $\alpha_j\alpha_j\alpha_j$, $\alpha_j\alpha_j\gamma_k$
 and $\alpha_i\alpha_j\gamma_k$ for some $j,k$ with $i\ne j$.}\]
Thus, we may assume that $\alpha_i\in M_2(r)$ for $i=1,2$.

Since $a_2(r)-2\geq 3$ and $a_2(r)=a_3(r)$, there exists $k$ such that $\gamma_k$ appears only once in $A_3(r)$.

We claim that
$\gamma_k\alpha_i\alpha_i\notin A_3(r)$ for $i=1,2$.
Let $\{x,y\}\in E_{\gamma_k}$.
If $\gamma_k\alpha_1\alpha_1\in A_3(r)$, then all edges from $\{x,y\}$ to others are beleonged to $E_{\alpha_1}$,
which implies
\[\mbox{$\gamma_i\alpha_1\alpha_1, \alpha_1\alpha_1\alpha_2\in A_3(r)$ for each $i=1,2,\ldots, a_2(r)-2$}.\]
Since $9\leq n$, we can take $V\in \binom{X}{4}$ with $V\cap \{x,y\}=\emptyset$ such that
$A_2(r_V)\subseteq \{\alpha_1,\alpha_2\}$.
This implies that $\alpha_i\alpha_i\alpha_i\in A_3(r)$ for some $i=1,2$.
Since $\alpha_2\in M_2(r)$ and
\[|\{\gamma_j\alpha_1\alpha_1, \alpha_1\alpha_1\alpha_2, \alpha_i\alpha_i\alpha_i\mid j=1,2,\ldots, a_2(r)-2\}|=a_2(r),\]
it follows that $i=2$. This implies that  $\{\alpha_2, \gamma_j\mid j=1,2,\ldots\}$ is closed, and
$(X,E_{\alpha_2}\cup \bigcup_iE_{\gamma_i})$ is a disjoint union of at least three cliques, 
so that $\alpha_1\alpha_1\alpha_1\in A_3(r)$, a contradiction to $a_2(r)=a_3(r)$.

By the claim,
\begin{equation}\label{eq:b2}
\mbox{$\gamma_k\alpha_1\alpha_2\in A_3(r)$.}
\end{equation}
Since $\alpha_1,\alpha_2\in M_2(r)$, $\alpha_i\alpha_i\beta_i\in A_3(r)$ 
for some $\beta_i\in A_2(r)$.
Since these elements are non-isometric in $(X,r)$, it follows from $a_2(r)=a_3(r)$ that
each of $\gamma_i$s appears only once in $A_3(r)$. Therefore, we conclude from (\ref{eq:b2}) that
\begin{equation}\label{eq:b3}
\mbox{$\gamma_j\alpha_1\alpha_2\in A_3(r)$ for each $j=1,2,\ldots, a_2(r)-2$.}
\end{equation}
Notice that
\[|\{\alpha_i\alpha_i\beta_i,\gamma_j\alpha_1\alpha_2,\mid i=1,2, j=1,2,\ldots, a_2(r)-2\}|=a_2(r).\]
It follows from (\ref{eq:b3}) that
\[\beta_1,\beta_2\in \{\alpha_1,\alpha_2\}.\]
Since $9\leq n$, it follows from (\ref{eq:b3}) that
there exists $W\in \binom{R_{\alpha_i}(x)}{3}$ for some $i=1,2$
such that $A_2(r_W)\subseteq \{\alpha_1,\alpha_2\}$.
This implies that
\[\mbox{$\alpha_i\alpha_i\alpha_i\in A_3(r)$ for some $i=1,2$.}\]
Notice that $\{\alpha_i\}$ is closed, so that
$(X,E_{\alpha_i})$ is a disjoint union of complete graphs by Lemma~\ref{lem:closed},
and it has exactly two connected components, otherwise,  the following non-isometric elements would induce a contradiction:
\[\mbox{$\alpha_j\alpha_j\alpha_i$, $\alpha_j\alpha_j\alpha_j$, $\alpha_i\alpha_i\alpha_i$ where $i\ne j$.}\]
Therefore, $(X,r)$ is a partition given in Theorem~\ref{thm:1}(ii).
\end{proof}

\section*{Acknowledgement}
This work was supported by a 2-Year Research Grant of Pusan National University.

%%%%%%%%%%%%%%%%%%%%%%%%%%%%%%%%%%%%%%%%%%%%%%%%%%%%%%%%%%%%%%%%%%%%%%%%%%%%%%%%%%%%%%%%%%%%%%%%

\end{document}